\documentclass[11pt,reqno,tbtags]{amsart}
\pdfoutput=1
\usepackage[]{amsmath,amssymb,amsfonts,latexsym,amsthm,enumerate}
\usepackage{amssymb}
\usepackage{graphicx}
\usepackage[numeric,initials,nobysame]{amsrefs}

\numberwithin{equation}{section}

\newtheorem{maintheorem}{Theorem}
\newtheorem{maincoro}[maintheorem]{Corollary}

\newtheorem*{conjecture*}{Conjecture}
\newtheorem{theorem}{Theorem}[section]
\newtheorem*{theorem*}{Theorem}
\newtheorem{lemma}[theorem]{Lemma}

\newtheorem{observation}[theorem]{Observation}

\newtheorem{corollary}[theorem]{Corollary}

\newtheorem*{definition*}{Definition}

\theoremstyle{definition}{

\newtheorem*{remark*}{Remark}
\newtheorem*{remarks*}{Remarks}
}

\newcommand{\R}{\mathbb R}
\newcommand{\C}{\mathbb C}
\newcommand{\Q}{\mathbb Q}

\newcommand{\Z}{\mathbb Z}
\newcommand{\F}{\mathbb F}

\renewcommand{\P}{\mathbb{P}}
\DeclareMathOperator{\var}{Var}

\newcommand{\whp}{\ensuremath{\text{\bf whp}}}

\renewcommand{\epsilon}{\varepsilon}

\newcommand{\deq}{\stackrel{\scriptscriptstyle\triangle}{=}}
\newcommand{\sq}{\textsc{Squares}}
\newcommand{\SP}[1][\mathbb{Z}]{\texttt{SP}_{#1}}

\newcommand{\gsum}[1][G]{\stackrel{\mbox{\tiny$#1$}}+}
\newcommand{\Gprod}[1][G]{\stackrel{\mbox{\tiny$#1$}}\times}

\newcommand{\vol}{\operatorname{vol}}

\newcommand{\greedy}{\textsc{Greedy}}

\date{}

\begin{document}
\title{Sums and products along sparse graphs}

\author{Noga Alon}
\address{Noga Alon\hfill\break
Sackler School of Mathematics and Blavatnik School of Computer
Science\\
Tel Aviv University\\
Tel Aviv, 69978, Israel,
and\hfill\break
Microsoft-Israel R\&D Center\\
Herzeliya, 46725, Israel.}
\email{nogaa@tau.ac.il}
\urladdr{}

\thanks{Research of N.\ Alon was supported in part by a USA Israeli
BSF grant, by a grant from
the Israel Science Foundation, by an ERC Advanced Grant
and by the Hermann Minkowski Minerva
Center for Geometry at Tel Aviv University.}

\author{Omer Angel}
\address{Omer Angel\hfill\break
Department of Mathematics\\
University of British Columbia\\
Vancouver, BC V6T-1Z2, Canada.}
\email{angel@math.ubc.ca}
\urladdr{}

\thanks{Research of O.\ Angel supported by NSERC and the University of Toronto}

\author{Itai Benjamini}
\address{Itai Benjamini\hfill\break
Weizmann Institute\\
Rehovot, 76100, Israel.}
\email{itai.benjamini@weizmann.ac.il}
\urladdr{}

\author{Eyal Lubetzky}
\address{Eyal Lubetzky\hfill\break
Microsoft Research\\
One Microsoft Way\\
Redmond, WA 98052-6399, USA.}
\email{eyal@microsoft.com}
\urladdr{}

\begin{abstract}
  In their seminal paper from 1983, Erd\H{o}s and Szemer\'edi showed that any $n$ distinct integers induce either $n^{1+\epsilon}$ distinct sums
  of pairs or that many distinct products, and conjectured a lower bound of $n^{2-o(1)}$. They further proposed a generalization of this problem, in which the sums and products are taken along the edges of a given graph $G$ on $n$ labeled vertices. They conjectured a version of the sum-product theorem for general graphs that have at least
  $n^{1+\epsilon}$ edges.

  In this work, we consider sum-product theorems for sparse graphs, and show that this problem has important consequences already when $G$ is a matching (i.e., $n/2$ disjoint edges): Any lower bound of the form $n^{1/2+\delta}$ for its sum-product over the integers implies a lower bound of $n^{1+\delta}$ for the original Erd\H{o}s-Szemer\'edi problem.

  In contrast, over the reals the minimal sum-product for the matching is $\Theta(\sqrt{n})$, hence this approach has the potential of achieving lower bounds specialized to the integers. We proceed to give lower and upper bounds for this problem in different settings. In addition, we provide tight bounds for sums along expanders.

  A key element in our proofs is a reduction from the sum-product of a matching to the maximum number of translates of a set of integers into the perfect squares. This problem was originally studied by Euler, and we obtain a stronger form of Euler's result using elliptic curve analysis.
\end{abstract}

\maketitle

\vspace{-0.3in}

\section{Introduction}

\subsection{Sums and products}
Let $A$ be a set of elements of some ring $R$. The \emph{sum-set} of $A$, denoted by $A + A$, and the \emph{product-set} of $A$, denoted by $A \times A$, are defined to be
\begin{align*}
 A + A &\deq \{x+y : x,y\in A\}~, &
 A\times A &\deq\{x\cdot y : x,y\in A\}~.
\end{align*}
The \emph{sum-product phenomenon} states that, in various settings, every set $A$ has either a ``large'' sum-set or a ``large'' product-set.
The intensive study of this area was pioneered by Erd\H{o}s and Szemer\'edi in their celebrated paper \cite{ES} from 1983, studying
sum-products over the integers. They showed that for some fixed $\epsilon > 0$, any set $A \subset \Z$ has $\max\{|A+A|,|A\times A|\} \geq |A|^{1+\epsilon}$, and conjectured that in fact $\epsilon$ can be taken arbitrarily close to $1$.

The Erd\H{o}s-Szemer\'edi sum-product problem for a set $A \subset \Z$ remains open, despite the considerable amount of attention it has received. The value of $\epsilon$ in the above statement was improved by Nathanson \cite{Nathanson} to $\frac1{31}$, by Ford \cite{Ford} to $\frac{1}{15}$ and by Chen \cite{Chen} to $\frac{1}{5}$. In 1997, a beautiful proof of Elekes \cite{Elekes} yielded a version of the sums and products over the reals (where the exponent is also believed to be $2-o(1)$) with $\epsilon=\frac14$, via an elegant application of the Szemer\'edi-Trotter Theorem. Following this approach, Solymosi \cites{Solymosi05,Solymosi} improved the bound on $\epsilon$ for sums and products over $\R$ to $\frac3{11}-o(1)$ and finally to $\frac13-o(1)$. For more details on this problem and related results, cf.\
 \cites{Chang03a,Chang03b,Chang08,Chang07,CS,ER} as well as \cite{VT} and the references therein.

Different variants of the sum-product problem were the focus of extensive research in the past decade, with numerous applications in Analysis,
Combinatorics, Computer Science, Geometry, Group Theory. Most notable is the version of the sum-product theorem for finite fields \cite{BKT}, where one must
add a restriction that $A \subset \F_p$ is not too large (e.g., almost all of $\F_p$) nor too small (e.g., a subfield of $\F_p$). See, e.g., \cites{Bourgain05,Bourgain07a,Garaev} for further information on sum-product theorems over finite fields and their applications, and also \cites{Bourgain07b,Tao} for such theorems over more general rings.

\subsection{Sums and products along a graph}
In their aforementioned paper \cite{ES}, Erd\H{o}s and Szemer\'edi introduced the following generalization of the sum-product problem,
where an underlying geometry (in the form of a graph on $n$ labeled vertices) restricts the set of the pairs used to produce the sums and products.
Formally, we let the sum-product with respect to a ring $R$
be the following graph parameter:
\begin{definition*}[Sum-product of a graph]
  Let $G=(V,E)$ be a graph. Given $A = (a_u)_{u\in V}$, an injective map of $V$ into some ring $R$, define the sum-set $A \gsum A$ and the product-set $A \Gprod A$ as follows:
\begin{align*}
A \gsum A &\deq \{a_u + a_v : u v \in E\}~, &
A \Gprod A &\deq \{a_u \cdot a_v : u v \in E\}~.
\end{align*}
The \emph{sum-product} of $G$ over $R$, denoted by \emph{$\SP[R](G)$}, is the smallest possible value of \mbox{$\max\{|A\gsum A| \;,\;|A \Gprod A |\}$} over all injections $A:V\to R$.
\end{definition*}

In other words, each ordered set $A \subset R$ of cardinality $|V(G)|$ is
associated with a sum-set and a product-set according to $G$ as follows:
The elements of $A$ correspond to the vertices, and we only consider sums
and products along the edges. Thus the original sum-product
 problem of Erd\H{o}s-Szemer\'edi corresponds to $\SP[\Z](K_n)$, where $K_n$ is the complete graph on $n$ vertices.

In the above notation, Erd\H{o}s and Szemer\'edi conjectured the following:
\begin{conjecture*}[Erd\H{o}s-Szemer\'edi \cite{ES}] For all $\alpha,\epsilon>0$ and every sufficiently large $n$, if $G=(V,E)$ is a graph on $n$ vertices satisfying $|E| \geq n^{1+\alpha}$ then
\begin{equation}
  \label{eq-sum-product-ineq}
  \mbox{\em$\SP[\Z](G) \geq |E|^{1-\epsilon}$}~.
\end{equation}
\end{conjecture*}

Erd\H{o}s and Szemer\'edi note that the above conjecture is likely to hold also over the reals.
However, for sparse graphs (graphs that contain $O(|V|)$ edges)
 there is a fundamental difference between the sum-product behavior over the integers and the reals. As stated in \cite{ES}, Erd\H{o}s had originally thought that \eqref{eq-sum-product-ineq} also holds when $G$ is a graph on $n$ vertices with at least $c n$ edges for some $c >0$. It was then shown by A.\ Rubin that the analogue of \eqref{eq-sum-product-ineq} for sparse graphs does not hold over $\R$, yet the question of whether or not it holds over $\Z$ remains open. See \cite{Chang04}, where the author relates this question of Erd\H{o}s to a famous conjecture of W.\ Rudin \cite{Rudin}.




Clearly, for any ring $R$ and graph $G=(V,E)$ we have that $\SP[R](G) \leq |E|$. We will mostly be interested in a choice of either $\Z$ or $\R$ for the ring $R$, and as we later state, these satisfy
\begin{equation}
  \label{eq-sqrtE-lower-bound}
\SP[\Z](G) \geq \SP[\R](G) \geq \sqrt{|E|}\quad\mbox{ for any graph $G=(V,E)$}~.
\end{equation}
Thus, when $G$ is a sparse graph with $n$ edges, the order of $\SP[\Z](G)$ is between $\sqrt{n}$ and $n$.
Our main focus in this paper is the case where $G$ is a matching, i.e., a graph consisting of disjoint edges.
The sum-product problem corresponding to this graph over the integers is already challenging, and as the next theorem demonstrates,
it has an immediate implication for the original Erd\H{o}s-Szemer\'edi problem:

\begin{maintheorem}\label{thm-Kn-reduction}
  Let $M$ be a matching of size $n$. The following holds:
  \begin{align}
    \mbox{\em$\SP[\Z]$}(M) &= O\left(\mbox{\em $\SP[\Z]$}(K_n) / \sqrt{n}\right)~,\label{eq-upper-bound-Kn}\\
    \mbox{\em$\SP[\Z]$}(M) &\leq n / \log(n)^\epsilon~\mbox{ for some $\epsilon > 0$}~.\label{eq-upper-bound-n/polylog}
  \end{align}
  In particular, if the sum-product of $M$ over $\Z$ is $\Omega(n^{1/2+\delta})$ for some $\delta > 0$,
   then every $n$-element subset $A \subset \Z$ satisfies $\max\{|A+A|,|A\times A|\} \geq\Omega(n^{1+\delta})$.
\end{maintheorem}

Note that \eqref{eq-upper-bound-Kn} translates any nontrivial lower bound in the sparse setting to one for the dense setting.
In particular, the best-possible lower bound of $n^{1-o(1)}$ for a matching of size $n$ would
imply that $\SP[\Z](K_n) \geq n^{3/2 - o(1)}$, improving upon the currently best known bound of $n^{4/3-o(1)}$.
Moreover, \eqref{eq-upper-bound-n/polylog} points out a relation between the upper bounds in these two settings:
Just as in the sum-product problem for the complete graph,
 the $\epsilon$ in the upper bound of $n^{1-\epsilon}$ for the sum-product of a matching is essential.

\subsection{Sum-products and Euler's problem on translates of squares}
Our next main result reduces the problem of obtaining a lower bound on $\SP[\Z](M)$, the sum-product of a matching over the integers, to bounding the maximum possible number of translates of a set of integers into the set of perfect squares, denoted by $\sq \deq \{ z^2 : z \in \Z\}$.
A special case of this problem was studied by Euler \cite{Euler}, and as we soon state, this problem fully captures the notion of a nontrivial lower bound on $\SP[\Z](M)$.

\begin{definition*}[Square translates]
  Let $F_k$ denote the maximum number of translates of a set $A$ that are contained within the set of
  perfect squares, taken over every $k$-element
  subset $A \subset \Z$:
  \begin{equation}
    \label{eq-Fk-def}
    F_k \deq \max_{A \subset \Z~,~|A|=k} \#\{ x : A + x \subset \sq\}~.
  \end{equation}
  Further let $F_k(n)$ denote this maximum with the added constraint that $|a| \leq n$ for all $a\in A$:
  \begin{equation}
    \label{eq-Fk(n)-def}
    F_k(n) \deq \max_{\substack{A \subset \{-n,\ldots,n\} \\ |A| = k}}
    \#\{ x\in\{-n,\ldots,n\} ~:~ A + x \subset \sq\}~.
  \end{equation}
\end{definition*}

Recall that $\sqrt{|E|}$ is a lower bound on $\SP[\R](G)$ for any graph $G=(V,E)$. The following theorem shows that, while this bound is tight for
a matching $M$ over $\R$, there is an equivalence between a nontrivial lower bound for $\SP[\Z](M)$
and a uniform upper bound on $F_k$ for some integer $k$.
\begin{maintheorem}\label{thm-Euler-reduction}
  Let $M$ be a matching of size $n$. The following holds:
  \begin{enumerate}
    \item \label{item-root-n-tight} We have {\em$\SP[\R](M) = \lceil\; \sqrt{n}\; \rceil$}.
    \item \label{item-Fk-root-n-bound} If $F_k = \infty$ for any $k$, then {\em$\SP[\Z](M) = \lceil\;\sqrt{n}\;\rceil$} for all $n$.
    \item \label{item-Fk-lower-bound} Conversely, if $F_k < \infty$ for some $k$, then
    {\em$\SP[\Z](M) = \Omega\big(n^{k/(2k-1)}\big)$}. Furthermore,
    for any $t=t(n)$ and any $A \subset \{-t,\ldots,t\}$ we have
    $$\max\{|A \gsum[M] A|\;,\;|A \Gprod[M] A|\}= \Omega\big( n^{k/(2k-1)} [F_k(4t^2)]^{-1/(2k-1)}\big)~.$$
  \end{enumerate}
\end{maintheorem}

As a corollary of the above theorem, we obtain a nontrivial lower bound of $n^{2/3}$
in case the elements of $A$ are all polynomial in $n$.

\begin{maincoro}\label{coro-poly-elements-2/3}
  Let $M$ be a matching of size $n$, and $A$ be a mapping of its vertices
  to distinct integers, such that $|a_v| \leq n^{O(1)}$ for all $v\in M$.
  Then $\max\{|A\gsum[M] A|\;,\;|A\Gprod[M] A|\} \geq n^{2/3-o(1)}$.
\end{maincoro}

In fact, the statement of Corollary~\ref{coro-poly-elements-2/3} holds as long as $|a| \leq n^{c \log\log n}$ for all $a \in A$
and some constant $c > 0$.

Euler \cite{Euler}*{Chapter 2.XIV, Article 223} studied translates of sets of three integers
into the set of perfect squares, corresponding to the parameter $F_3$. He provided examples where nontrivial translates
exist, and showed how to find such translates in general if they are
known to exist. In Section~\ref{sec:square-translates} we extend Euler's results, and use elliptic curves to construct sets of three integers for which there are infinitely many such translates ($F_3 = \infty$).

The parameter $F_4$, together with the results of Theorem~\ref{thm-Euler-reduction}, enables us to deduce
another lower bound on $\SP[\Z](M)$, assuming a major conjecture in arithmetic geometry
 --- the Bombieri-Lang conjecture for rational points on varieties of general type.

\medskip

\begin{maincoro}\label{coro-bombieri-lang-4/7}
Assume the Bombieri-Lang conjecture, and let $M$ denote a matching of size $n$.
Then {\em$\SP[\Z](M) = \Omega(n^{4/7})$}.
\end{maincoro}

Notice that, combining the above lower bound (assuming the Bombieri-Lang conjecture) with Theorem~\ref{thm-Kn-reduction} yields
a lower bound of $\epsilon=\frac{1}{14}$ for the sum-product of the complete graph. While this does not improve the best known exponent for $\SP[\Z](K_n)$, as we later state, it does improve all known sum-product bounds for graphs with slightly smaller degrees (e.g., of average degree $n^{1-\delta}$ for certain small $\delta>0$).

See, for instance, \cites{BGP,CG} for other implications of the Bombieri-Lang conjecture on problems involving the perfect squares.

\subsection{Sums along expander graphs}
Up till now, we considered sums and products along graphs, where each of the sum-set and product-set could be small (yet they could not both
be small at the same time): For instance, the sum-set along a matching can consist of a single element.
The final part of this paper investigates the smallest possible size of the sum-set along other underlying geometries. Note that this problem is trivial for dense graphs, as the maximal degree of a graph $G$ is clearly
a lower bound on $|A\gsum A|$.

As we later explain, a straightforward extension of one of our
basic
arguments for the sum-product of a matching gives that, for instance,
if $G$ is a vertex-transitive graph on $n$ vertices with odd-girth $\ell$ then
\begin{equation}
  \label{eq-odd-girth-lower-bound}
|A\gsum A| \geq n^{1/\ell}\qquad\mbox{ for any $A \subset \R$}~.
\end{equation}
In particular, when $G$ is a disjoint union of triangles, $|A\gsum A| \geq n^{1/3}$, and we later show that
this bound is tight.
It is natural to assume that the sum-set along $G$ should be forced to be larger if $G$ had, in some sense, stronger
interactions between its vertices, and specifically, if the graph is an \emph{expander} (defined below).
Surprisingly, our results show that the sum-set along an $n$-vertex expander can be of size only $O(\log n)$, and this is best possible.

The \emph{conductance} of a graph $G=(V,E)$, denoted by $\Phi(G)$, is defined as
\begin{equation}
  \label{eq-Phi-def}
\Phi(G) \deq \mathop{\min_{S\subset V}}_{\vol(S) \neq 0} \frac{e(S,\overline{S})}{\min\{\vol(S)\,,\,\vol(\overline{S})\}}~,
\end{equation}
where $\overline{A}$, $\vol(A)$ and $e(A,B)$ denote the complement of $A$, its volume (the sum of its degrees)
and the number of edges
between $A$ and $B$ respectively.
For a real $\delta > 0$ and a graph $G$ without isolated vertices, we say that $G$ is a $\delta$-(edge)-expander if $\Phi(G) > \delta$.
For further information on these objects and their numerous applications, cf., e.g., \cite{HLW}.

\bigskip

The next theorem characterizes the smallest possible cardinality of the sum-set of $A \subset \Z$ along an expander.
\begin{maintheorem}\label{thm-expander} For any $0 < \delta < \frac12$ there exist constants $C,c > 0$ such that:
\begin{enumerate}[1.]
  \item If $G$ is a $\delta$-expander on $n$ vertices then
  $$|A \gsum A| \geq c \log n~\qquad\mbox{ for any $A\subset \Z$ , $|A|=n$}~.$$
  \item There exists a regular $\delta$-expander $G$ on $n$ vertices such that
  $$|A\gsum A|\leq C \log n~\qquad\mbox{ for $A = \{1,2,\ldots,n\}$}~.$$
\end{enumerate}
\end{maintheorem}

\subsection{Organization}
The rest of this paper is organized as follows.
Section~\ref{sec:upper-bound-matching} contains the proof of Theorem~\ref{thm-Kn-reduction}, which provides upper bounds on $\SP[\Z](M)$ and relates it to $\SP[\Z](K_n)$. In Section~\ref{sec:lower-bound-matching} we prove Theorem~\ref{thm-Euler-reduction}, which gives lower bounds for $\SP[\Z](M)$ in terms of the parameters $F_k$ (translates of $k$ integers into the squares).
Section~\ref{sec:square-translates} focuses on this problem of translates of a set into the squares: We
first analyze $F_3$, and extend Euler's result using elliptic curves. We then discuss $F_4$
and its implication on the sum-product of the matchings.
In Section~\ref{sec:sum-sets} we study sum-sets along
other geometries, and prove Theorem~\ref{thm-expander}, which establishes tight bounds for sums along expanders.
The final section, Section~\ref{sec:conclusion}, contains
concluding remarks and open problems.

\section{Upper bounds for the sum-product of a matching}\label{sec:upper-bound-matching}

In this section, we prove Theorem~\ref{thm-Kn-reduction}, which provides upper bounds for the
sum-product of a matching over the integers. Throughout this section, let $M$ denote a matching consisting of $n$ disjoint edges.
We begin with a simple lemma.

\begin{lemma}\label{lem-find-matching}
  Suppose the edges of a graph $G=(V,E)$ are properly coloured with $k$
  colours, and let $\Delta$ be the maximal degree in $G$. Then $G$ contains
  a matching of at least $|E|/(4\Delta)$ edges involving at most
  $k/(2\Delta)$ colours.
\end{lemma}

\begin{proof}
  Repeatedly select all edges of the most used colour and delete all edges
  adjacent to them, until at least $|E|/(4\Delta)$ edges have been
  selected. Up to that point at most $|E|/2$ edges are deleted, so at each
  step at least $|E|/(2k)$ edges are selected. Thus the number of steps is at
  most $2k/(4\Delta)$.
\end{proof}

\subsection{A sub-linear upper bound: proof of \eqref{eq-upper-bound-n/polylog}}
The desired upper bound given in inequality \eqref{eq-upper-bound-n/polylog} is equivalent
to the following statement: There is a fixed $\epsilon > 0$ so that, for every sufficiently
large $n$,
there exists an ordered set $A$ of $2n$ distinct integers satisfying
$$|A \gsum[M] A| \leq \frac{n}{(\log n)^{\epsilon}}\quad\mbox{ and }\quad|A \Gprod[M] A| \leq \frac{n}{(\log n)^{\epsilon}}~.$$
We need the following  result of Erd\H{o}s \cite{Erdos}.

\begin{lemma}
\label{lem-erdos}
There is a fixed $\epsilon>0$ such that for every sufficiently
large $N$, the
number of integers  which are the product of two integers, each no
greater than $N$, is smaller than $N^2/[128(\log N)^{2\epsilon}]$.
\end{lemma}

Let $N = 16n$ be a large integer, let
$I$ be the interval of all integers in
$$\Big[N-\frac{N}{32 (\log N)^{\epsilon}},
N+\frac{N}{32 (\log N)^{\epsilon}}\Big)~,$$ and let $G_0=(V,E)$ be the graph
on the set of vertices  $\{1,2, \ldots ,N\}$ in which $i$ and $j$
are connected iff $i+j \in I$.  Note that every vertex of $G_0$ has
degree at least $d/2$ and at most $d$, where
$d=\frac{N}{16 (\log N)^{\epsilon}}$, and in particular $|E| \geq Nd/4$.
Assign each edge of $G_0$ a colour according to the product of its
endpoints, and note that $G_0$ is now properly coloured with at most
$k=\frac{N^2}{128(\log N)^{2\epsilon}}$ colours, due to Lemma~\ref{lem-erdos}.

By Lemma~\ref{lem-find-matching}, there is a matching in $G_0$ consisting
of at least $|E|/(4d) \geq N/16$ edges which are coloured by at most
$k/(2d) = \frac{N}{16 (\log N)^{\epsilon}}$ distinct colours. Thus we have
found $N/16$ disjoint pairs of integers with at most $\frac{N}{16 (\log
  N)^{\epsilon}}$ distinct sums and as many products.
\qed

\subsection{From matchings to dense graphs: proof of \eqref{eq-upper-bound-Kn}}
We prove a stronger statement than the one given in Theorem~\ref{thm-Kn-reduction}, and bound $\SP[\Z](M)$ in terms of the sum-product
of any sufficiently dense graph (rather than the complete graph). This is formalized by the following theorem.

\begin{theorem}
  \label{thm-sparse-dense}
  Let $G=(V,E)$ be a graph on $N$ vertices with maximum degree at most $D
  \geq 10 (\log N)^2$ and average degree at least $d$, such that $N$ is
  large enough and $d \geq 5\sqrt{D}$. Suppose that $S \leq
  \frac{Nd}{16 \sqrt D}$ and that $A=\{a_v : v \in V\}$ are distinct
  integers satisfying
  \[
  |A \gsum A| \leq S \quad\mbox{ and }\quad |A \Gprod A| \leq S~.
  \]
  Then there is a matching $M$ of $n=\frac{Nd}{32 D}$ edges in $G$,
  so that
  \[
  |A \gsum[M] A| \leq \frac{2S}{\sqrt D} \quad\mbox{ and }\quad |A
  \Gprod[M] A| \leq \frac{2S}{\sqrt D}~.
  \]
\end{theorem}

\begin{proof}
Fix $p=\frac{1}{\sqrt D}$, and let $R$ be a random subset of
$A \gsum A$ obtained by picking every element $s \in A \gsum A$,
randomly and independently, with probability $p$. Let $H$
be the spanning subgraph of $G$ consisting of all edges  $uv$
so that $a_u+a_v \in R$.

By standard large deviation estimates for binomial distributions, with high
probability the total size of $R$ is smaller than $2 S p = 2 S/\sqrt D$,
and the maximum degree in $H$ is smaller than $2 D p = 2 \sqrt D$. (Note
that the degree of each vertex in $H$ is indeed a binomial random variable,
as each edge of $G$ incident with the vertex remains in $H$ randomly and
independently with probability $p$.) Moreover, we claim that the number of
edges of $H$ is at least $\frac{N d}{4 \sqrt D}$ with probability at least
$1/5$, hence with positive probability $H$ satisfies all of these
conditions. To see this last claim, let $m_i\leq N/2$ be the number of
edges in $G$ with sum $i$. Then
\[
\var |E(H)| = p(1-p) \sum_i m_i^2 < p \sum_i \frac{N}{2} m_i = \frac{N
  |E(G)|}{2\sqrt{D}}.
\]
By Chebyshev's inequality, $\P(|E(H)| < p|E(G)|/2) \leq
\frac{4\sqrt{D}}{d} \leq 4/5$, implying the claim.

Fix a choice of $H$ for which the above conditions hold. Assign to each
edge $e = u v$ of $H$, a colour given by the numbers associated to its
endpoints: $a_u \cdot a_v$. Note that this is a proper colouring of $H$
with at most $S$ colours.

Applying Lemma~\ref{lem-find-matching} to $H$, yields a matching in $H$
consisting of at least $\frac{N d}{32 D}$ edges, with at most
$\frac{S}{4\sqrt{D}}$ colours. This matching gives $\frac{N d}{32 D}$ pairs
of integers with at most $\frac{2S}{\sqrt{D}}$ sums and
$\frac{S}{4\sqrt{D}}$ products, as required.
\end{proof}



\begin{remark*}
The assumption $D \geq 10 (\log N)^2$ can be easily relaxed, as it is not
essential that all degrees in $H$ will be at most $2 \sqrt D$, it suffices
to ensure that no set of $\frac{Nd}{16D}$ vertices captures more than
$\frac{Nd}{8 \sqrt D}$ edges. It is also not difficult to prove a version
of the above theorem starting with the assumption that
$|A \gsum A| \leq S$ and $|A \Gprod A| \leq T$,
where $S$ and $T$ are not necessarily equal. Similarly, the requirement
$d\geq 5\sqrt{D}$ can be relaxed (if one accepts larger sum and product
sets) by splitting the edges with a given sum into
subsets for the construction of $H$.
\end{remark*}

An immediate application of the last theorem is the following.
\begin{corollary}\label{cor-sparse-dense-bound}
If $G$ is a $D$-regular graph on $N$ vertices with $D\geq 10(\log N)^2$
and there exists a set of $N$ distinct integers $A$ so that
$|A \gsum A| \leq S$ and $|A \Gprod A| \leq S$, then there
is a matching $M$ of size $n=\frac{N}{32}$ and  a set $B$ of $2n$
distinct integers so that $|B \gsum[M] B| \leq \frac{2S}{\sqrt D}$
and  $|B \Gprod[M] B| \leq \frac{2S}{\sqrt D}$.
\end{corollary}
In particular, for
$G$ being a complete graph this implies that if there is a set
$A$ of $N$ distinct integers so that $|A +A| \leq S$ and $|A \times A| \leq S$,
then there is  a matching $M$ of size $\Omega(N)$ and a set $B$ of
$2|M|$ distinct integers so that $$|B \gsum[M] B| \leq O(S/\sqrt N)
\quad\mbox{ and }\quad|B \Gprod[M] B| \leq O(S/\sqrt N)~.$$ This proves \eqref{eq-upper-bound-Kn}, and completes
the proof of Theorem~\ref{thm-Kn-reduction}. \qed

\section{Lower bounds for matchings and translates into squares}\label{sec:lower-bound-matching}

In this section, we prove Theorem~\ref{thm-Euler-reduction}, which relates lower bounds for the sum-product of
the matching over the integers to the square-translates problem defined in the introduction.

\begin{proof}[\emph{\textbf{Proof of Theorem~\ref{thm-Euler-reduction}}}]
We first elaborate on inequality \eqref{eq-sqrtE-lower-bound}, which stated that
any graph $G=(V,E)$ satisfies $\SP[\Z](G) \geq \SP[\R](G) \geq \sqrt{|E|}$.
This follows immediately from the next simple observation:
\begin{observation}\label{obs-sumprod-lower-bound}
Let $\F$ be a field and $G=(V,E)$. Then any injection $A : V \to \F$ satisfies
$|A \gsum A | \cdot |A \Gprod A| \geq |E|$.
\end{observation}

Indeed, since any quadratic polynomial over $\F$ has at most $2$ roots, any
two elements $\{x,y\} \in \F$ are uniquely determined by their sum $s=x+y$
and their product $p=xy$. In particular, when the characteristic of $\F$ is
other than $2$,
\begin{equation}
  \label{eq-xy-sols}
  \{x,y\} = \Big\{\frac{s \pm \sqrt{s^2 - 4p}}{2} \Big\}~.
\end{equation}
The above bound is tight (up to rounding) whenever it is possible to take a square-root of elements in $\F$ (in fact, a slightly weaker condition already suffices). To demonstrates this over $\R$ and any $n = m^2$ for $m\geq 1$, let $X$ denote a set of $m$ reals chosen uniformly from the interval $[5,6]$. Clearly, every pair $s,p\in X$ satisfies
$$ s^2 - 4p \geq 25 - 24 > 0~,$$
and furthermore, with probability $1$ there exist $2n$ distinct solutions to the $m^2$ equations of the form \eqref{eq-xy-sols}, as $s,p$ range over all possible values in $X$.
This shows that
$$ \SP[\R](M) = \lceil \,\sqrt{n}\, \rceil~,$$
even with the added constraint $|A \gsum A | = |A \Gprod A|$. Next, we wish to relate $\SP[\Z](M)$ to the parameters $F_k$, defined in \eqref{eq-Fk-def}.

\bigskip

To prove Item~\eqref{item-Fk-root-n-bound} of the theorem, assume that indeed $F_k =\infty$ for all $k$.
We need to show that if $M$ is a matching comprising $n=m^2$ edges, then there exists a set $A$ of $2n$ distinct integers such that both $|A\gsum[M] A| = m$ and $|A\gsum[M] A| = m$.

Set $K = 2m^3$. By the assumption on $\{F_k\}$ we have $F_{m+1} > K$.
In particular, there exist two sets of distinct integers,
$X=\{x_1,\ldots,x_K\}$ and $Y=\{y_0, y_1,\ldots,y_m\}$, such that
\[ x+y \in \sq\quad\mbox{ for all $x\in X$ and $y\in Y$}. \]
By translating $X,Y$ in opposite directions (recall that
$F_{m+1} \geq K+1$) we may assume that $0=y_0 \in Y$, and so
$X\subset\sq$. We may also assume
$4\mid x_i$ for all $i$, (otherwise, multiply
$X$ and $Y$ by $4$), and set $\tilde{X} = \{\sqrt{x} : x \in X\} \subset \Z$.

Let $P = \{-\frac14 y_1,\ldots,-\frac14 y_m\}$. We claim that there exists a subset $S \subset \tilde{X}$ of size $m$, such that all the solutions to \eqref{eq-xy-sols} with $s\in S$, $p\in P$ are distinct.

To see this, first notice that if $s \in \tilde{X}$ and $p \neq p' \in P$, then
$s^2-4p$ and $s^2-4p'$ are two distinct squares by our assumption on $X$ and $Y$. It thus follows that there are $2m$ distinct
 solutions to \eqref{eq-xy-sols} for this $s$ and all $p \in P$.
Let $A_s$ denote this set of $2m$ integer solutions.

Consider the graph $H$ on the vertex set $\tilde{X}$, where two distinct vertices $s,s' \in \tilde{X}$ are adjacent if and only if they share a common solution
to \eqref{eq-xy-sols}, that is, if $A_s \cap A_{s'}$ is nonempty.

Next, note that any $a \in \Z$ (in fact even in $\R$)
and $p \in P$ can correspond to at most one possible value of $s$
such that $a$ is a solution of \eqref{eq-xy-sols} with this pair $(s,p)$
(namely, the only possible value for $s$ is $a+\frac{p}{a}$, where
here we used the fact that $p \neq 0$ for all $p \in P$).
It then follows that the degree of any $s\in \tilde{X}$ in $H$ is at most
$$ |A_s||P| - 1 < 2m^2~.$$
In other words, $H$ is graph on $K$ vertices with maximal degree less than $2m^2$, and thus has an independent set
(an induced subgraph containing no edges)
of size at least $K/(2m^2) = m$. Furthermore, such an independent set can easily be found via the \greedy\ algorithm (sequentially
processing $\tilde{X}$ and adding vertices that are not incident to the current induced subgraph).

%
%
%

Combined with Observation~\ref{obs-sumprod-lower-bound}, this implies that
$\SP[\Z](M) = m$, proving Item~\eqref{item-Fk-root-n-bound}.

%

\bigskip

It remains to prove Item~\eqref{item-Fk-lower-bound}. Let $A = \{a_u : u \in V\}$ be a set of $|V|$ distinct integers, let $S = A\gsum[M] A$ and $P = A\Gprod[M] A$ denote the sum-set and product-set of $A$ along $M$ resp., and set $m \deq \max\{|S|,|P|\}$. Define the following $m\times m$ binary matrix $B$, indexed by the elements of $S$ and $P$
(if either $S$ or $P$ has less than $m$ elements, $B$ may have all-zero rows or columns respectively):
\begin{equation}
  \label{eq-B-def}
  B_{s,p} = \left\{\begin{array}
    {ll}
1 & s=a_u + a_v\mbox{ and }p=a_u a_v\mbox{ for some }e=(u,v)\in E~,\\
0 &\mbox{otherwise}~.
  \end{array}\right.
\end{equation}
By definition there are two distinct integer solutions to \eqref{eq-xy-sols} for any $(s,p)$ such that $B_{s,p} = 1$.
In particular,
\begin{equation}
  \label{eq-B-property}
   s^2-4p \in \sq~\mbox{ for any $(s,p)$ such that $B_{s,p}=1$}~.
\end{equation}

Let $t=t(n)$, and consider $F_k(4t^2)$, defined in \eqref{eq-Fk(n)-def} as the maximum number of translates
that $k$ integers $\{a_1,\ldots,a_k\} \subset \{-(2t)^2,\ldots,(2t)^2\}$ can have into the set of perfect squares.
It then follows from \eqref{eq-B-def} and \eqref{eq-B-property} that, if $|a_u| \leq t$ for all $u\in V$, then there
are at most $r \deq F_2(4t^2)$ translates of any set $\{s_1^2,\ldots,s_k^2\}$ with $s_1,\dots,s_k\in S$
into the squares. Similarly, there are at most $r$ translates of any set $\{-4p_1,\dots,-4p_k\}$ with $p_1,\dots,p_k\in P$
into the squares. It follows that $B$ does not contain a $k\times(r+1)$ minor consisting of all 1's. Equivalently, $B$ represents a bipartite graph $G$ with color classes of size $m$ each, which has $e(M) = n$ edges and does not contain a copy of the subgraph $K_{k,r+1}$.

The case $k=2$ is somewhat simpler and has interesting
consequences, and so we deal with it first. In what follows we
need a special case of a well known result of K\"ov\'ari, S\'os and
Tur\'an. For completeness, we reproduce its (simple) proof.
Let $N(u)$ and $d(u)$ denote the neighborhood of a vertex $u$ and its degree resp., and further
let $N(u,v)$ and $d(u,v)$ denote the common neighborhood of two vertices $u,v$ and its size (the co-degree) respectively. According
to these notations, a standard calculation shows that the  total of all co-degrees in $G$ is
\begin{align*}
 D &\deq \sum_{u\in V(G)}\binom{d(u)}2 = \frac12 \sum_u (d(u))^2 - e(G) \\
 &\geq \frac{\left(\sum_u d(u)\right)^2}{4m} - e(G) = \frac{e(G)^2}m - e(G)~,
\end{align*}
where the inequality was due to Cauchy-Schwartz (recalling that $G$ is a graph on $2m$ vertices).
Dividing by $\binom{2m}2$ and using the fact that $e(G) \leq \binom{2m}2$, we obtain that the average co-degree in $G$ is at least
$$\frac{D}{\binom{2m}2} \geq \frac{e(G)^2}{2m^3} - 1 ~.$$
On the other hand, as $G$ contains no $K_{2,r+1}$, this quantity is necessarily at most $r$, and so
$$ r \geq \frac{e(G)^2}{2m^3}-1 = \frac{n^2}{2m^3}-1~.$$
Rearranging, we have
$$ m \geq \Big(\frac{n^2}{2(r+1)}\Big)^{1/3}~,$$
which by definition of $m,r$ gives that for all $A \subset\{-t,\ldots,t\}$
\begin{equation}
  \label{eq-F2-lower-bound}
  \max\{|A \gsum[M] A|\;,\;|A \Gprod[M] A|\}= \Omega\big( n^{2/3} [F_2(4t^2)]^{-1/3}\big)~,
\end{equation}
thus proving Item~\eqref{item-Fk-lower-bound} of the theorem for $k=2$.

Note that at this point we can infer Corollary~\ref{coro-poly-elements-2/3}.
Indeed, letting $a,b\in \Z$, any $x\in \Z$ that translates $\{a,b\}$ into the squares satisfies
\begin{align*}
  a+x &= y_1^2~\mbox{ and }~ b+x = y_2^2~\mbox{ for some $y_1,y_2\in\Z$}~,
\end{align*}
and so
$$ a - b = y_1^2 - y_2^2 = (y_1-y_2)(y_1+y_2)~.$$
It follows that the number of such translates corresponds to the number of divisors of $a-b$.
As it is well known that the number of divisors
of an integer $N$ is at most
$\exp\big[O(\frac{\log N}{\log\log N})\big] \leq N^{o(1)}$,
we deduce that
$$ F_2(N) \leq N^{o(1)}~.$$
Combining this with \eqref{eq-F2-lower-bound} immediately implies the required lower bound $\max\{|A \gsum[M] A|\;,\;|A \Gprod[M] A|\} \geq n^{2/3-o(1)}$ whenever every $a\in A$ has $|a|\leq n^{O(1)}$.

To generalize the lower bound to any fixed $k$, we apply
the general theorem of K\"ov\'ari, S\'{o}s and Tur\'{a}n \cite{KST}
on the density of binary matrices without certain
sub-matrices consisting only of $1$ entries.
We use the following version of this theorem (see, e.g., \cite{Jukna}*{Chapter 2.2}, and also \cite{Matousek}):

\begin{theorem}[K\"ov\'ari-S\'{o}s-Tur\'{a}n]
Let $k \leq r $ be two integers, and let $G$ be a bipartite graph with $m$ vertices in each of its parts.
If $G$ does not contain $K_{k,r}$ as a subgraph, then
$$e(G)\leq (r-1)^{1/k} (m-k+1)m^{1-1/k} + (k-1)m~.$$
\end{theorem}

As noted above, with $r = F_k(4t^2)$, for any $k$ rows of $B$ there can be at most $r$
columns forming a sub-matrix consisting only of $1$ entries,
and vice versa. Equivalently, the bipartite graph $G$ does not contain $K_{k,r+1}$
as a subgraph. Recalling that $e(G)=n$, we obtain that
$$ n \leq r^{1/k} m^{(2k-1)/k} + (k-1)m~. $$
Either $m>\frac{n}{2(k-1)}$, in which case we are done, or else this yields
$$ m \geq \left(n/2\right)^{k/(2k-1)} r^{-1/(2k-1)}~, $$
that is,
$$\max\{|A \gsum[M] A|\;,\;|A \Gprod[M] A|\}= \Omega\big( n^{k/(2k-1)} [F_k(t^2)]^{-1/(2k-1)}\big)~.$$
This concludes the proof of Theorem~\ref{thm-Euler-reduction}.
\end{proof}

\section{Translates of a set into the squares}\label{sec:square-translates}

\subsection{Euler's problem: translates of three integers into the squares}
In this section, we study the parameter $F_3$: We are interested in
integer solutions to the following set of $3$ equations in $4$ variables ($X,Y_1,Y_2,Y_3$):
\begin{equation}\label{eq:3-squares}
  Y_i^2 = X + a_i\qquad(i=1,2,3)~,
\end{equation}
where the $a_i$'s are distinct integers.
By clearing denominators, it is equivalent to consider rational solutions rather than integer ones.

Euler \cite{Euler}*{Chapter 2.XIV} studied this question in the following
form:
\begin{quote}
``To find a number, $x$, which, added to each of the given numbers, $a$,$b$,$c$, produces a square"
\end{quote}
After demonstrating that this is impossible in some families of parameters, he concludes (in the following $m=b-a$ and $n=c-a$):
\begin{quotation}
``\ldots it is not easy to
choose such numbers for $m$ and $n$ as will render the solution
possible. The only means of finding such values for $m$ and $n$ is to
imagine them, or to determine them by the following method.''
\end{quotation}
Euler's method is to start with a given solution (assuming one is available), and look for
others, using transformations of certain quartics into squares. He considers the integers $\{0,2,6\}$, so $\frac14$ is a solution, and proceeds to find the solution $\big(\frac{191}{60}\big)^2$. Euler further claims that this method
can be used recursively to find other solutions. Following his line of arguments
gives the following recursion relation: If $x^2$ is a solution for the integers $\{0,2,6\}$,
that is, $y^2=x^2+2$ and $z^2 = 6 + x^2$ for some $y,z\in\Q$, then
$$ x' = \frac{(x^4 - 12) (x+y)}{2 x y z \sqrt{2 + 2 x (x + y)}} = \frac{x^4 - 12}{2x y z}$$
also provides such a solution, since in that case it is easy to verify that
\begin{align*}
(x')^2 + 2 &= \frac{(x^4+4x^2+12)^2}{(2xyz)^2}~,&
(x')^2 + 6 &= \frac{(x^4+12x^2+12)^2}{(2xyz)^2}~.
\end{align*}
 Plugging in $x=\frac{191}{60}$ gives the additional solution $(x')^2$ for $x'=\frac{1175343361}{1154457480}$ (the next
element obtained via this recursion has 38-digit numerator and denominator). Euler does not discuss when
this method may guarantee an aperiodic series of translations (though this may be shown using similar
elementary methods).

In what follows, we present a general framework for
obtaining sets of 3 integers with infinitely many translates into the squares, using
elliptic curves. This approach further provides a quantitative lower bound on the number of translates,
 in terms of the height of the elements of the original set (i.e., a lower bound on $F_3(n)$).
 We begin by showing that indeed $F_3 = \infty$.

\begin{theorem}\label{thm-elliptic}
  The parameter $F_3$ is unbounded. Moreover, there exist distinct integers
  $\{a_1,a_2,a_3\}$ with infinitely many translates into the set of perfect
  rational squares.
\end{theorem}

\begin{proof}


We may assume that there is at least one solution to \eqref{eq:3-squares}, and without loss of generality $X=0$ is a
solution, so we have $\sqrt{a_i}\in\Q$ for all $i$ and can instead consider the equations
$$Y_1^2-Y_i^2=a_1-a_i\qquad(i=2,3)~.$$
For some $t,u\in \Q$ to be later specified, let
\begin{equation}
  \label{eq-Yi(u,t)}
  Y_1=\sqrt{a_1}+u~,~\quad Y_2=\sqrt{a_2}+t u~.
\end{equation}
It follows that
$$ Y_1^2 - Y_2^2 = a_1 - a_2 + \left( (1-t^2)u + 2(\sqrt{a_1}-t \sqrt{a_2})\right) u~,$$
thus if $t^2 \neq 1$ then $u = \frac{2(\sqrt{a_1}-t \sqrt{a_2})}{t^2 - 1}$ is the unique non-zero rational such that
\begin{equation*}
  Y_1^2 - Y_2^2 = a_1 - a_2
\end{equation*}
(Note that, for $t = \pm 1$, only $u = 0$ satisfies
this equality, since $a_1\neq a_2$). Using this substitution,
we find
$$Y_3^2 = Y_1^2 - (a_1 - a_3) = (\sqrt{a_1}+u)^2 - (a_1 - a_3) = \frac{Q(t)}{(t^2-1)^2}~,$$ where $Q(t)$ is a monic quartic with known
coefficients (derived from $a_1$, $a_2$, $a_3$).
Therefore, to solve \eqref{eq:3-squares} we
need $Q(t)$ to be a rational square.

Let $G(t)$ be a quadratic and $H(t)$ linear so that $Q(t)=G^2(t)+H(t)$. If
$Q(t)$ is a rational square, let
\begin{align}
  \label{eq-S-t-def}
  T_0 &\deq G(t)+\sqrt{Q(t)}~,   &   S_0 &\deq t T_0~.
\end{align}
We have
\begin{align*}
0 = Q(t) - G^2(t) - H(t) &= T_0(T_0-2G(t))-H(t)\\
&= T_0^2 - 2T_0 G(S_0/T_0) - H(S_0/T_0)~,
\end{align*}
which upon multiplying by $T_0$ becomes a polynomial relation between $S_0,T_0$.
Next, let $S,T$ be affine changes of $S_0,T_0$ as follows:
\begin{align*}
  T &= \tfrac12 T_0 + \frac{  3 a_1 a_2 - 2 a_1 a_3 - 2 a_2 a_3 + a_3^2 }{3 a_3^2}~,\\
S &= \tfrac12 S_0 + \frac{\sqrt{a_1 a_2}}{2a_3} T_0 -
    \frac{ (a_1 a_2)^{3/2} - \sqrt{a_1^3 a_2} a_3 - \sqrt{a_1 a_2^3} a_3 - \sqrt{a_1 a_2} a_3^2 }{a_3^3}~.
\end{align*}
This brings the polynomial relation between $S_0,T_0$ to an elliptic curve in
standard form:
$$S^2=T^3 + \alpha T + \beta ~,$$
where
\begin{align}\label{eq-curve-parameters}
\alpha &= \frac{-\sum_i a_i^2 + \sum_{i<j} a_i a_j}{3a_3^2} ~,~
\beta = \frac{2\sum_i a_i^3 - 3\sum_{i\neq j} a_i^2 a_j + 12 a_1 a_2 a_3}
{27a_3^3} ~.
\end{align}

We next show that $F_3 = \infty$. Consider the choice $a_1 = \frac49$, $a_2=\frac{16}9$, $a_3=\frac19$.
By \eqref{eq-curve-parameters}, this produces the elliptic curve
$$ S^2 = T^3 -63 T + 162~,$$
which has positive rank (namely, rank $1$, computed via SAGE). This gives rise to infinitely many rational points $(T,S)$, and
using \eqref{eq-S-t-def}, we can recover the value of $t = S/T$ from each of them.
Recalling that $u$ is uniquely determined by $t$, we now return to \eqref{eq-Yi(u,t)}
and obtain the rational points $Y_1,Y_2,Y_3$ from each pair $(T,S)$, as required.
\end{proof}

\begin{table}
$$
\mbox{\small$ \begin{array}{|r @{~\bullet\frac{\qquad}{}\bullet~} l|r|r|}
\hline
\multicolumn{2}{|c|}{\mbox{Integer assignment}} & \multicolumn{1}{|c|}{\mbox{Sum}} & \multicolumn{1}{|c|}{\mbox{Product}}\\
\hline
\rule[5mm]{0mm}{0mm}
   283815  &   17974425 & 4 \cdot 4564560 &    5101411431375\\
   597975  &    8531145 & 2 \cdot 4564560 &    5101411431375\\
  1954575  &    2609985 &         4564560 &    5101411431375\\
 -1711710  &   19969950 & 4 \cdot 4564560 &  -34182763114500\\
 -2852850  &   11981970 & 2 \cdot 4564560 &  -34182763114500\\
 -3993990  &    8558550 &         4564560 &  -34182763114500\\
 -6607744  &   24865984 & 4 \cdot 4564560 & -164308056580096\\
 -9042176  &   18171296 & 2 \cdot 4564560 & -164308056580096\\
-10737584  &   15302144 &         4564560 & -164308056580096
\rule[-3mm]{0mm}{0mm} \\
\hline
\end{array}$}
$$
\caption{Optimal sum-product mapping for a matching of size 9 over the
  integers: 3 sums and 3 products, found using the elliptic curve $S^2 =
  T^3 -63 T + 162$.}\label{tab-opt-9}
\end{table}

Table~\ref{tab-opt-9} demonstrates how the above analysis provides an optimal family of $9$ pairs of distinct integers,
inducing only $3$ sums and $3$ products:
$$\SP[\Z](M) = 3\quad\mbox{ when $M$ is the matching on $9$ edges}~,$$
where the lower bound follows from \eqref{eq-sqrtE-lower-bound}.

\begin{remark*}
  In general, the above analysis leads, for any integer $n$, to an explicit
  construction giving a family of $3n$ pairs of distinct integers, inducing
  $3$ distinct sums and $n$ distinct products.
\end{remark*}

\begin{remark*}
  An alternative way for proving Theorem~\ref{thm-elliptic} is to consider
  the curve $y^2 = (x+a_1)(x+a_2)(x+a_3)$. This curve contains the rational points of the curve defined in \eqref{eq-curve-parameters},
  and one may obtain infinitely many of them by starting from one of the points and repeatedly doubling it.
\end{remark*}

\subsection{A quantitative lower bound}
Recall that $F_2(n) = \exp\big[\Theta\big(\frac{\log n}{\log\log n}\big)\big]$,
which implies that $$F_3(n) \leq F_2(n) \leq n^{o(1)}~.$$ The next theorem provides a lower bound on $F_3(n)$:
\begin{theorem}
The function $F_3$ satisfies $F_3(n) = \Omega\left((\log n)^{5/7}\right)$.
\end{theorem}

\begin{proof}
We need the following well-known facts concerning elliptic curves.

The Mordell-Weil Theorem states that, for any elliptic curve $E(\Q)$,
the group of rational points on the curve is
a finitely generated abelian group: $E(\Q)\cong E_\mathrm{torsion} \bigoplus \Z^r$, where $E_\mathrm{torsion}$ is the \emph{torsion group} (points of finite order) and $r$ is the \emph{rank} of the curve.
Mazur's Theorem characterizes the torsion group of any $E(\Q)$ as one of $15$ given (small) groups.


The \emph{logarithmic height} of a rational $x=\frac{p}q$, denoted by $h(x)$, is defined as
$$h(p/q) \deq
\max\{\log|p|\;,\;\log|q|\}~.$$ For a point $P=(x,y)\in E(\Q)$ we let $h(P)\deq h(x)$.
Further define the \emph{canonical height} of $P \in E(\Q)$ to be
$$\hat{h}(P) \deq \frac12 \lim_{m\to\infty} \frac{h(mP)}{m^2}~.$$

The following theorem states some well known properties of the canonical height:
\begin{theorem}[canonical height]
  The following holds:
  \begin{enumerate}[1.]
  \item The canonical height $\hat{h}$ is quadratic and satisfies the parallelogram law:
  $$\hat{h}(mP) = m^2 \cdot \hat{h}(P)~\mbox{ and }~\hat{h}(P+Q)+\hat{h}(P-Q) =
    2(\hat{h}(P)+\hat{h}(Q))~.$$
  \item The canonical height $\hat{h}$ is roughly logarithmic:
  $$\big| \hat{h}(P) - h(P) \big| < K~\mbox{ for some $K = K(E)$}~.$$
  \end{enumerate}
\end{theorem}

By the above facts, we can now deduce the following corollary
for the number of rational points on $E(\Q)$ with a given bound
on their numerators and denominators:

\begin{corollary}\label{cor-small-points}
  In an elliptic curve $E(\Q)$ of rank $r$, the number of points with
  $\hat{h}(P)<M$ has order $M^{r/2}$. Consequently this is also the number of
  points with numerator and denominator bounded by $O(\exp(M))$.
\end{corollary}

To see this, take a basis $P_1,\ldots,P_r$ for the abelian group $E(\Q)$. It is now
straightforward to verify that the requisite
points are (up to the slight effect of the torsion group) the points $\sum_{i=1}^r a_i P_i$ where $a_i<\sqrt{M}$.

Let $a_1,a_2,a_3$ be distinct rational points, and let $r$ denote the rank
of the elliptic curve defined in \eqref{eq-curve-parameters}. By the above discussion, there
are $M^{r/2}$ rational solutions with denominators bounded by $O(\exp(M))$. Clearing denominators
results in $M^{r/2}$ integer solutions to the system corresponding to $a_1',a_2',a_3' \in \Z$, where all absolute values are at most $N = O(\exp(M^{1+r/2}))$. Equivalently,
$M = \Omega\left((\log N)^{2/(r+2)}\right)$, giving the following estimate on $F_3(n)$:
$$ F_3(n) \geq \Omega( \left(\log n\right)^{r/(r+2)} )~.$$
In particular, one can verify that a choice of $a_1 = 3$, $a_2=34$, $a_3=89$
(obtained by a computer search using SAGE) for the $a_i$'s produces
a curve of rank $5$, implying the desired result.
\end{proof}

\begin{figure}[t]
\centering
\fbox{\includegraphics{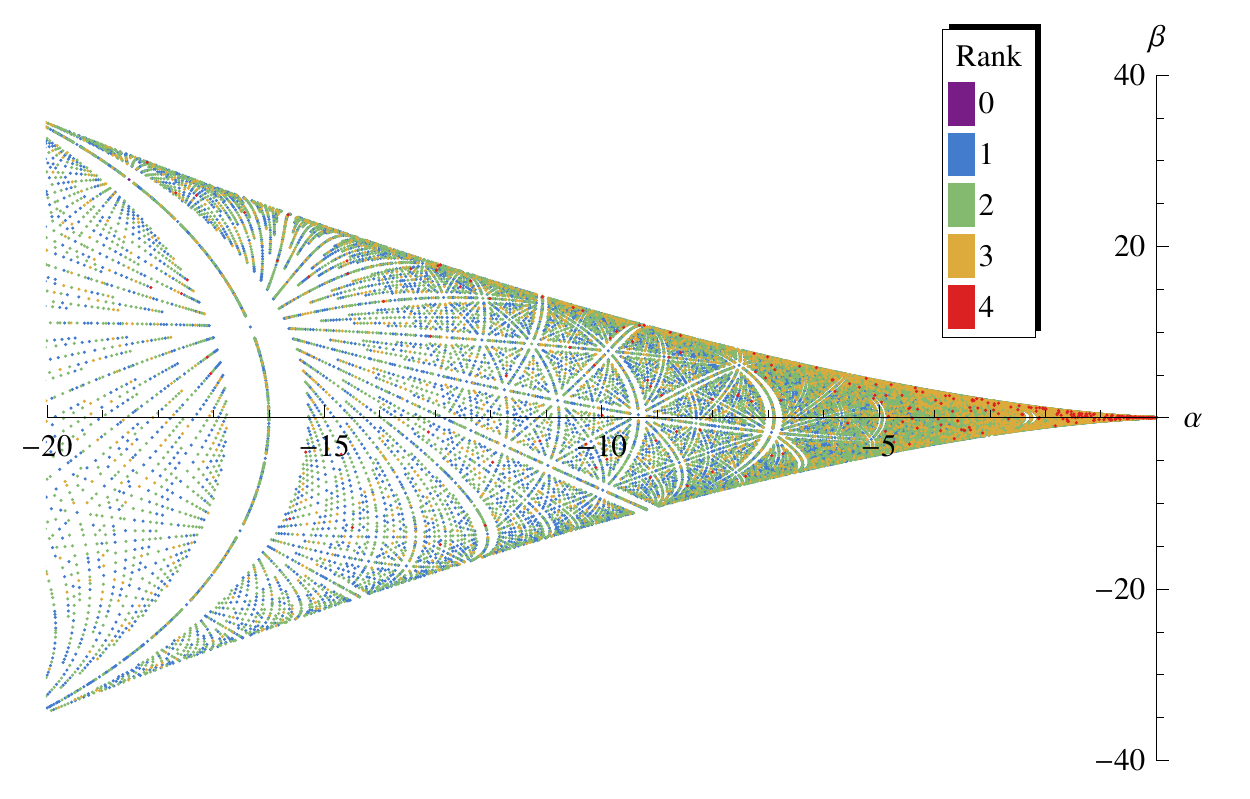}}
\caption{Ranks of all curves corresponding to triplets $\{a_1,a_2,a_3\}\subset \{0,1,\ldots,100\}$, as given
by \eqref{eq-curve-parameters}.}
\label{fig:elliptic}
\end{figure}

\begin{remark*}
The above curve of rank $5$ was found by a computer search (see Figure~\ref{fig:elliptic}). Not every
  elliptic curve can be represented by \eqref{eq-curve-parameters}, and it
  is even unknown if there are elliptic curves of arbitrarily large rank.

It appears that the curves obtained from \eqref{eq-curve-parameters} have rank 0 only when
$a_1,a_2,a_3$ are of the form $xy,xy+x^2,xy+y^2$ for some $x,y$.
The rank distribution for curves obtained from
  \eqref{eq-curve-parameters} appears to concentrate on rank $1$ as the height of the $a_i$'s tends to infinity
(e.g., random samples of curves from the ranges $\{L,\ldots,2L\}$ with $L\in\{10^3,10^4,10^5\}$ gave rank $1$ in about $0.33$, $0.48$ and $0.70$ fraction of the samples resp.).
This is in contrast with all
  elliptic curves, where it is believed that the rank is $0$ and $1$ with
  density $1/2$ each.

Further note that the constants implicit in Corollary~\ref{cor-small-points} depend on the
  curve. A uniform (in the curve) bound on the number of points of height
  at most $m$ in an elliptic curve may lead to an improved upper bound.
\end{remark*}

\subsection{Translates of four or more integers and curves of higher genus}\label{sec:4-square-translates}
We next turn our attention to the parameters $F_k$ for $k > 3$.
Recalling the general framework of the problem, we are interested in integer solutions to the set of $k$ equations
in the $k+1$ variables $(X,Y_1,\ldots,Y_k)$:
\begin{equation}\label{eq:k-squares}
Y_i^2 = X + a_i \qquad (i=1,\ldots,k)~,
\end{equation}
where the $a_i$'s are assumed to be distinct
coefficients.

Let $\mathcal{S}\subset\C^{k+1}$ denote the set of all complex solutions.
By adding a point at infinity, $\mathcal{S}$ can be
compactified into a one (complex)
dimensional manifold.

\begin{lemma}\label{lem-genus}
  The genus of $\mathcal{S}$ is $1+(k-3)2^{k-2}$.
\end{lemma}

\begin{proof}
  Using the relation to the Euler characteristic $\chi(\mathcal{S}) = 2-2g(\mathcal{S})$, it
  suffices to show that
   $$\chi(\mathcal{S}) = (3-k)2^{k-1}~.$$ This is achieved by means of the
  Riemann-Hurwitz formula: 
For a map $\pi:\mathcal{S}\to \mathcal{S}'$
  which is $N\to 1$ except at some ramification points of $\mathcal{S}'$ we have
  \[
  \chi(\mathcal{S}) = N\chi(\mathcal{S}') + \sum e_p-c_p
  \]
  where the sum is over the ramification points and $e_p$ is the
  ramification index at $p$ and $c_p$ the cycle index.

  We apply this to our manifold $\mathcal{S}$ and the Riemann sphere $\mathcal{S}'$, and with
  the map $\pi(X,Y_1,\dots,Y_k) = X$. Any $X$ has $2^k$ pre-images with
  coordinates given by the square roots of $X+a_i$. The ramification points
  are
  $X=-a_i$ for each $i$ and $X=\infty$. At $X=a_i$ we have only $2^{k-1}$
  pre-images (here we use that all $a_i$'s are distinct) and so
  $$e_{-a_i}-c_{-a_i} = 2^k-2^{k-1}=2^{k-1}~.$$ The singularity at $X=\infty$
  is of the same type ($2^{k-1}$ coinciding points of ramification index 2,
  and thus $e_\infty-c_\infty=2^{k-1}$). Combining these we find
  \[
  \chi(\mathcal{S}) = 2^k\cdot 2 - k 2^{k-1} - 2^{k-1} = (3-k)2^{k-1},
  \]
  as required.
\end{proof}

Note that for $k=3$, the genus of $\mathcal{S}$ is 1 and thus $\mathcal{S}$ is
an elliptic curve, as we have already seen in the above analysis of this case.
For $k>3$ the genus is larger, and the understanding of rational points
on $\mathcal{S}$ is relatively scant.


It is well known that rational points in curves of high genus are uncommon: Indeed,
Falting's Theorem states that the number of rational points on any curve of genus $g > 1$
is finite. The following result of Caporaso, Harris and Mazur further states that
this quantity is uniformly bounded from above, if we accept a major conjecture in
Arithmetic Geometry.

\begin{theorem}[Caporaso-Harris-Mazur \cite{CHM}]\label{thm-CHM}
Assume the Bombieri-Lang conjecture. Then for any $g > 1$ there is some constant $B(g)$
  such that the number of rational points on any curve of genus $g$ is at
  most $B(g)$.
\end{theorem}

Combining this with Lemma~\ref{lem-genus} implies that, if we accept the Bombieri-Lang
conjecture, then for any $k \geq 4$ there is some constant $B$ depending only on $k$
such that the number of solutions to \eqref{eq:k-squares} is at most $B$.
In other words, if the Bombieri-Lang conjecture holds, then $F_k \leq B(1+(k-3)2^{k-2}) < \infty$ for any $k \geq 4$.
Together with Theorem~\ref{thm-Euler-reduction} (Part~\eqref{item-Fk-lower-bound}), this proves Corollary~\ref{coro-bombieri-lang-4/7}.

\begin{remark*}
The curve determined by \eqref{eq:k-squares} is not completely general, and
so it may be possible to get bounds on the number of solutions without
assuming the Bombieri-Lang conjecture. One approach is to multiply the
equations and note that $P(X) := \prod (X+a_i) = (\prod Y_i)^2$ is a
square, where $P(X)$ is some polynomial of degree $k$ with distinct integer
roots. This determines a curve of lower genus, but still of genus greater
than $1$ for $k>4$.
\end{remark*}

\subsection{Consequences for the original Erd\H{o}s-Szemer\'edi problem}
We have shown that, assuming a plausible conjecture in Number Theory,
for every matching $M$ of size $n$
and every set $B$ of $2n$ distinct integers,
$$
\max\{|B \gsum[M] B|, |B \Gprod[M] B|\}  \geq \Omega(n^{4/7}).
$$
Together with Theorem~\ref{thm-sparse-dense}, this implies that if $G$ is any graph
on $n$ vertices and at least $n^2/k$ edges, and $A$ is any set of $n$ distinct
integers, then
\begin{equation}
\label{eq-n/k-4/7}
\max\{|A \gsum A|, |A \Gprod A|\}  \geq
\Omega\Big(\min\Big\{\frac{n^{15/14}}{k^{4/7}}~,~\frac{n^{3/2}}k\Big\} \Big).
\end{equation}
Indeed, if $S=\max\{|A \gsum A|\;,\; |A \Gprod A|\}$ then,
by Theorem \ref{thm-sparse-dense} with $D=n$ and $d=n/k$, either $S = \Omega(n^{3/2}/k)$, or there is a matching $M$ of size
$\Omega(n/k)$ and a set of of $2|M|$ distinct integers $B$ so that
$$
\Omega\left((n/k)^{4/7}\right) \leq \max\{|B \gsum[M] B|\;,\; |B \Gprod[M] B|\}
\leq O\Big(\frac{S}{\sqrt n}\Big),
$$
supplying the desired lower bound for $S$.

It is worth noting that without
assuming any unproven conjectures, one can prove
that for every $G$ as above, and every set $A$ of positive integers
\begin{equation}
  \label{eq-best-known-bnd}
\max\{|A \gsum A|\;,\; |A \Gprod A|\}  \geq
\Omega\bigg(\frac{n^{10/9-o(1)}}{k^{19/9-o(1)}}\bigg)~.
\end{equation}  This can be done as follows.
Suppose $$\max\{|A \gsum A|\;,\; |A \Gprod A|\} =cn~,$$ where $G$ has
$n$ vertices and at least $n^2/k$ edges. By the proof of Gowers \cite{Gowers}
of the Balog-Szemer\'edi  Theorem \cite{BS} (see also \cite{SSV}), there
are two subsets $A',B' \subset A$ so that
\begin{align*}
&|A'|=|B'| \geq \Omega(n/k)~\quad\mbox{ and}\\
&|A'+B'| \leq O(c^3 k^5 n)~,~ |A' \times B'| \leq O(c^3 k^5 n)~.
\end{align*}
However, if we plug in the best current lower bound for 
the sum-product of the complete bipartite graph,
due to Solymosi \cite{Solymosi}, we have that
$$\max\{|A'+B'|\;,\; |A' \times B'|\} \geq \Omega((n/k)^{4/3-o(1)})~,$$
implying that
$c^3 \geq \Omega(n^{1/3-o(1)}/k^{19/9+o(1)})$, which gives the desired estimate
$$cn \geq \Omega\bigg(\frac{n^{10/9-o(1)}}{k^{19/9+o(1)}}\bigg)~.$$ Note that for $k>n^{1/19}$
this does not give any nontrivial bound.

On the other hand, the estimate \eqref{eq-n/k-4/7} (which depends on the validity of the Bombieri-Lang conjecture)
gives a nontrivial bound for all $k<n^{1/8}$. Furthermore, our bound improves upon \eqref{eq-best-known-bnd}
already for $k > n^{5/194}$.

\section{Sums along expanders}\label{sec:sum-sets}

In this section, we study sum-sets along various underlying geometries,
and obtain tight bounds for the case of expander graphs (Theorem~\ref{thm-expander}).
Note that we no longer consider the product-set along the graph.

\subsection{Lower bound for sums along general graphs}

We begin with a straightforward lower bound for the size of the sum-set along a given
graph:
\begin{observation}\label{obs-sum-lower-bound}
Let $\F$ be a field of characteristic $\mathrm{char}(\F)\neq 2$. Then
for any graph $G$ and an injective map $A$ from its vertices to $\F$ we have
$$ |A \gsum A| \geq \left[2k \cdot e_{k}(G)\right]^{1/k}~ \mbox{ for any odd integer $k\geq 3$}~,$$
where $e_k(G)$ denotes the number of cycles of length $k$ in $G$.
\end{observation}

To prove this, let $A$ be a mapping from the vertices of $G$ to $\F$, and
consider a cycle
$$ u = v_0, v_1,\ldots, v_{k} = u ~,$$
where $v_0 v_{k} \in E$ and $v_i v_{i+1} \in E$ for $i=0,\ldots,k-1$.
We then have that $$2A_u = \sum_{i=1}^{k} (-1)^{i-1} (A_{v_{i-1}}+A_{v_{i}})~,$$
and so $A_u$ is uniquely determined by the sums on the edges of the cycle.
Therefore, for any directed cycle as above (fixing the starting point $u=v_0$ and the orientation), we must have a different
sequence of $k$ sums along the edges (otherwise, for two cycles
starting at $u\neq u'$ we would get $A_u = A_{u'}$, while for the two orientations of the same cycle we would get $A_{v_1}=A_{v_{k-1}}$).
Altogether, any undirected cycle gives rise to $2k$ distinct sequences of sums (accounting for both orientations).
This implies the desired lower bound.

Note that we cannot infer a bound on $|A\gsum A|$ in terms of
$e_k(G)$ when $\mathrm{char}(\F)=2$ or $k$ is even.  To see this,
suppose one cycle has labels $a_1,\dots,a_k$. Disjoint cycles may
then be labeled $a_i+(-1)^i x$ with an arbitrary $x$ to get the
same $k$ sums. It is thus possible to choose at least $|\F|/k^2$ values
of $x$ (and in fact even more, with a careful choice of the values
$a_1, \ldots ,a_k$),
so that the labels are all distinct, whereas $|A \gsum A| = k$.

The lower bound of Observation~\ref{obs-sum-lower-bound} can be asymptotically tight: To demonstrate this for $k=3$, we consider
the graph $G$ comprising $\binom{m}3$ disjoint triangles, and construct for it an injection $A$ such that $$|A \gsum A| = m~.$$
Let $S$ be a set of $m$ distinct sums, all even, to be specified later. We assign each of the triangles a different triplet of these sums. This determines the integer values at each of the vertices. To conclude the construction, we need $S$ to yield distinct integer values at different vertices; this is achieved, for instance, by choosing
$$ S = \{ 2^i : i=1,\ldots,m\}~,$$
since $2^{i_1}+2^{j_1}-2^{k_1} \neq 2^{i_2}+2^{j_2}-2^{k_2}$ for any two sets $\{i_1,j_1,k_1\} \neq \{i_2,j_2,k_2\}$
(alternatively, one can obtain $A$ which is injective \whp, by choosing the elements of $S$ independently and uniformly over some large ground set).

It is not difficult to extend this example and show that Observation~\ref{obs-sum-lower-bound} gives
the optimal order of $|A\gsum A|$ whenever $G$ is a disjoint union of cycles of odd length $k$.

\subsection{Tight bounds for sums along expanders}
We now prove Theorem~\ref{thm-expander}, showing that
the sum-set along an $n$-vertex expander can be of size $O(\log n)$,
and this is best possible.

Recall the definition of an expander given in the introduction.
In the special case where the graph $G=(V,E)$ is $d$-regular, we
call $G$ a $\delta$-expander if for every set  $X$ of at most
$|V|/2$ vertices, the number of edges from $X$ to its complement
is at least $\delta d |X|$.

An $(n,d, \lambda)$-graph is a connected $d$-regular graph on $n$ vertices,
in which the absolute value of  each nontrivial eigenvalue is at
most $\lambda$. This notion was introduced by the first author in the
1980's, motivated by the observation that such graphs in which
$\lambda$ is much smaller than $d$ exhibit strong pseudo-random
properties. In particular, it is easy to show (see, e.g.,
\cite{AM}) that
\begin{equation}
\label{e51}
\mbox{every $(n,d, \lambda)$-graph is a $\delta$-expander for }~ \delta=\frac{d-\lambda}{2d}~.
\end{equation}
\begin{theorem}
\label{t51}
For every fixed $\delta<\frac12$ there is a constant $c=c(\delta)$
so that the following holds: For any sufficiently large
$n$ there is a $\delta$-expander $G$ on $n$ vertices such that
$|A \gsum A| \leq c \log n$ for $A=\{1,2,\ldots ,n\}$.
\end{theorem}
\begin{proof}
For an abelian group $\Lambda$ and a subset $T \subset \Lambda$, the
{\em Cayley sum graph} $G=G(\Lambda,T)$
of $\Lambda$ with respect to $T$ is the graph
whose set of vertices is $\Lambda$, in which $yz$ is an edge for each
$y,z \in \Lambda$ satisfying $y+z \in T$. Clearly, this is a
$|T|$-regular graph.

Let $D$ be the adjacency matrix
of $G$. It is well known (cf., e.g., \cite{Alon})
that its eigenvalues can be expressed in terms of
$T$ and the characters of $\Lambda$. Indeed,
for every character $\chi$ of $\Lambda$ and every $y \in \Lambda$,
$$\left(D \chi\right)(y) = \sum_{s \in T} \chi(s-y)=\Big(\sum_{s \in T} \chi(s)\Big)
\overline{\chi(y)}~.$$ Applying $D$ again, it follows that
$$D^2{\chi} = \Big|\sum_{s \in T}
\chi(s)\Big|^2 \chi~.$$
Therefore, the eigenvalues of the symmetric
matrix $D^2$ are  precisely the expressions
$|\sum_{s \in T} \chi(s)|^2$, where the characters are the
corresponding eigenvectors, and as the characters are orthogonal,
these are all eigenvalues.  It then follows that each nontrivial
eigenvalue of the graph $G=G(\Lambda,T)$ is, in absolute value,
$|\sum_{s \in T} \chi(s)|$ for some nontrivial character $\chi$
of $\Lambda$ (it is not difficult to  determine the signs as well,
but these are not needed here).

In particular, for the additive
group $\Z_n$ and for
$T \subset \Z_n$, every nontrivial eigenvalue  of
the Cayley graph of $\Z_n$ with respect to $T$ is, in absolute value,
$|\sum_{s \in T} \omega^s|$, where $\omega$ is a nontrivial $n$-th
root of unity.
The following lemma is proved in \cite{Alon} by a simple probabilistic
argument.
\begin{lemma}[\cite{Alon}]
\label{l52}
For every integer $d \leq n^{2/3}$ there exists a subset $T \subset \Z_n$
of cardinality $d$ so that for every nontrivial $n$-th root of unity
$\omega$
$$
\Big|\sum_{s \in T} \omega^s\Big| \leq 3 \sqrt d \sqrt {\log (10n)}.
$$
\end{lemma}
To complete the proof of Theorem~\ref{t51},
assume $n$ is sufficiently large as a function
of $\delta$. Let $T$ satisfy the assertion of the lemma, with
$|T|=c' \log n$,
where $c'$ is chosen  to ensure that
$$
\frac{c' \log n - 3 \sqrt {c' \log n} \sqrt {\log (10n)}}{2c' \log n}
\geq \delta.
$$
Since $\delta<1/2$ it is obvious that $c'=c'(\delta)>0$ can be chosen
to satisfy this inequality, and thus the graph $G=G(\Z_n,T)$
has the desired
expansion properties, by (\ref{e51}).
Every sum of endpoints of an edge in
it is equal, modulo $n$, to a member of $T$, and as
we are considering addition over the integers, this means that
for $A=\{1,2, \ldots ,n\}$, $$|A \gsum A| \leq 2|T|=2c' \log n~,$$
as required.
\end{proof}

It remains to show that the logarithmic estimate is tight.
This is done  by considering the diameter of graphs
$G$ for which $A \gsum A=T$ is a relatively small set.
\begin{lemma}
\label{lem-sumset-diameter}
Let $G=(V,E)$ be a graph on $n$ vertices and $A=\{a_v:v \in V\}$
be a set of distinct integers. If $|A \gsum A| = s$ and the diameter
of $G$ is $r$, then
\begin{equation}\label{eq-sum-diameter-bnd}
2^{s} { \binom{r+s}{s} } \geq n/2~.
\end{equation}
In particular, if $r \leq L \log n$ for some $L > 1$ then $s \geq \Omega(\log_{L} n)$.
\end{lemma}
\begin{proof}
Put $T=A \gsum A$.
Fix a vertex $u \in V$. If $$u=v_0 v_1 v_2 \ldots v_l=w$$ is a path
of length $l$ in $G$ starting at $u$, then there  are
(not necessarily distinct) elements $g_1,g_2, \ldots ,g_l \in T$
so that $$a_w=g_l-g_{l-1}+g_{l-2}-  \cdots +(-1)^{l-1}g_1+(-1)^l a_u~.$$
It follows that for every $a_w \in A$, either the difference
$a_w-a_u$ or the sum $a_w+a_u$ can be expressed as the inner product of an
integral vector $x$ of length $s$ with $\ell_1$-norm at most $r$
with the vector
$(g:~ g \in T)$.  The number of choices for the vector $x$ is at most
$2^{s} \binom{r+s}{s} $, implying \eqref{eq-sum-diameter-bnd}.

The conclusion in case
$r \leq L \log n$ for some $L > 1$ now follows by a simple manipulation, since
$2^s \binom{r+s}s \leq \big(\frac{2\mathrm{e}(r+s)}s\big)^s $.
\end{proof}


\begin{corollary}
\label{cor-exp-sumset-lower-bound}
Let $G=(V,E)$ be a $\delta$-expander on $n$ vertices, and let $A=\{a_v:v \in V\}$
be a set of distinct integers.
Then $|A \gsum A| \geq \Omega(\frac{1}{\log (1/\delta)} \log n)$.
\end{corollary}
\begin{proof}
It is well-known (and easy) that if $G$ is a $\delta$-expander on $n$ vertices, its diameter satisfies
\begin{equation}
  \label{eq-expander-diam}
 \mathrm{diam}(G) \leq 2\Big\lceil\frac{|E|}{\log(1+\delta)}\Big\rceil=O\Big(\frac{\log n}{\delta}\Big)~.
\end{equation}
Indeed, by definition \eqref{eq-Phi-def} and the assumption $\Phi(G) \geq \delta$,
$$ e(S,\overline{S}) \geq \delta \vol(S)~\mbox{ for any $S\subset V$ with $0 \neq \vol(S) \leq \vol(\overline{S})$}~,$$
and it is straightforward to infer from this that
$$ \vol\left(S \cup N(S)\right) \geq (1+\delta)\vol(S)\quad\mbox{ for any $S$ with } 0 \neq \vol(S) \leq \vol(\overline{S})~,$$
implying \eqref{eq-expander-diam}. The desired result now follows by Lemma \ref{lem-sumset-diameter}.
\end{proof}

\section{Concluding remarks and open problems}\label{sec:conclusion}

We have introduced the study of sums and products along the edges
of sparse graphs, showing that it is related to the well
studied investigation of the problem for dense graphs, as well as
to classical problems and results in Number Theory.

There are many
possible variants of the problems considered here. In particular,
the results in
the previous section  suggest  the study of the minimum possible
value of  $|A \gsum A|$ for a given graph $G$, and its relation to
the structural properties of the graph. Lemma \ref{lem-sumset-diameter}
provides a lower bound for this quantity for graphs with a small
diameter, and Observation \ref{obs-sum-lower-bound} supplies
another lower bound in terms of the number of short odd cycles
in the graph (provided the characteristic of the field is not $2$).

In view of the relation to  the original conjecture of
Erd\H{o}s and Szemer\'edi for the dense case,  an
interesting open problem is whether the minimum possible value of
$\SP[\Z](M)$ for a matching $M$  of size $n$  is
$n^{1-o(1)}$.
We believe that this is the case, but a proof will certainly
require some additional ideas.

Finally, given its ramifications on the sum-product exponents of sparse
graphs, it would be interesting to establish whether the parameter $F_k$
(the maximum number of translates of $k$ integers into the perfect squares)
is indeed finite for some $k$.
\begin{conjecture*}
  There is a finite $k$ so that there are \emph{no} sets $X,Y\subset \Z$ of sizes $|X|=|Y|=k$
  that satisfy
  \begin{equation}\label{eq-x+y-sq}x + y \in \sq\mbox{ for all $x\in X$ and $y\in Y$}\,.\end{equation}
\end{conjecture*}
Note that, by our results, for all $k$ there exist sets $X,Y\subset \Z$ of sizes $|X|=3$, $|Y|=k$ that do satisfy \eqref{eq-x+y-sq}.

\section*{Acknowledgments}

We wish to thank Henry Cohn, Noam Elkies, Jordan Ellenberg, 
Moubariz Garaev, Andrew Granville, Patrick Ingram, 
Ram Murty, Joseph Silverman, J{\'o}zsef
Solymosi, Endre Szemer{\'e}di and Terence Tao for useful discussions. This
work was initiated while the second and third authors were visiting the
Theory Group of Microsoft Research.

\end{document}